\documentclass[12pt,reqno]{amsart}

\usepackage{hyperref}

\usepackage[headsep=30pt,footskip=30pt,bottom=2.5cm]{geometry}

\usepackage[all]{xy}

\xymatrixcolsep{2cm}

\numberwithin{equation}{section}

\allowdisplaybreaks[1]

\newtheorem{theorem}{Theorem}[section]

\newtheorem{proposition}[theorem]{Proposition}
\newtheorem{definition}[theorem]{Definition}

\newtheorem{lemma}[theorem]{Lemma}

\newtheorem{corollary}[theorem]{Corollary}

\theoremstyle{remark}

\newtheorem{remark}[theorem]{Remark}

\begin{document}

\title[Chen-Ruan cohomology and moduli of parabolic bundles]{Chen-Ruan
cohomology and moduli spaces of parabolic bundles over a Riemann surface}

\author[I. Biswas]{Indranil Biswas}

\address{School of mathematics, Tata Institute of Fundamental Research, 1 Homi Bhabha Road, Colaba, Mumbai 
400005, India}

\email{indranil@math.tifr.res.in}

\author[P. Das]{Pradeep Das}

\address{Department of Mathematical Sciences,
Rajiv Gandhi Institute of Petroleum Technology, Jais,
Amethi, Uttar Pradesh 229304, India}

\email{pradeepdas0411@gmail.com}

\author[A. Singh]{Anoop Singh}

\address{School of mathematics, Tata Institute of Fundamental Research, 1 Homi Bhabha Road, Colaba, Mumbai 
400005, India}

\email{anoops@math.tifr.res.in}



\begin{abstract}
Let $(X,\,D)$ be an $m$-pointed compact Riemann surface of 
genus at least $2$. For each $x \,\in\, D$, fix full flag 
and concentrated weight system $\alpha$. 
Let $P \mathcal{M}_{\xi}$ denote the moduli space
of semi-stable parabolic vector bundles
of rank $r$ and determinant $\xi$ over $X$ with weight system $\alpha$,
where $r$ is a prime number and $\xi$ is a holomorphic line bundle over $X$ of degree $d$
which is not a multiple of $r$.
We compute the Chen-Ruan cohomology of the orbifold for the action on $P \mathcal{M}_{\xi}$
of the group of $r$-torsion points in ${\rm Pic}^0(X)$.
\end{abstract}
\maketitle

\section{Introduction}\label{sec:introduction}

Chen and Ruan introduced a new cohomology theory in \cite{CR} for orbifolds which is named after them. It is the
degree zero part of the small quantum cohomology ring, constructed in \cite{CR1}, of the orbifold. It
contains the usual cohomology ring of the orbifold as a sub-ring.

The Chen-Ruan cohomology for the orbifolds arising from the moduli space of stable vector bundles of rank $2$
and degree $1$ over a compact Riemann surface was computed in \cite{BP}, which was
subsequently generalized for arbitrary prime rank in \cite{BP1}.

Let $X$ be a compact Riemann surface of genus $g\,\geq\, 2$, and let $D\,=\, \{x_1,\, \ldots,\, x_m \}$
be a finite subset of $X$; the points of $D$ will be called parabolic points. Fix a holomorphic
line bundle $\xi$ on $X$ of degree one.
Let $PM_{\xi}(2)$ denote the moduli space of full flag stable parabolic vector bundles $E_*$ over $X$ of rank
$2$ and fixed determinant $\bigwedge^2 E \,\cong\, \xi$; the parabolic weights are assumed to be generic.
The moduli space $PM_{\xi}(2)$ is a smooth projective variety of dimension $3g-3+m$.
Let $\Gamma_2 \,\subset\, {\rm Pic}^0(X)$ be the subgroup defined by the points of order two, meaning
the holomorphic line bundles $L$ with $L^{\otimes 2}\,=\,{\mathcal O}_X$. Then $\Gamma_2$ acts 
on $PM_{\xi}(2)$; the action of $L\, \in\, \Gamma_2$ sends any $E_*\, \in\, PM_{\xi}(2)$ to 
$E_* \otimes L$. We get the smooth orbifold $PM_{\xi}(2)/\Gamma_2$.
The Chen-Ruan cohomology of it was computed in \cite{BD}.

Here we consider the moduli space of parabolic bundles of prime rank $r\,\geq\, 3$. Fix a holomorphic
line bundle $\xi$ on $X$ whose degree is not a multiple of $r$. Let $P\mathcal{M}_\xi$ denote the
moduli space of stable parabolic bundles on $X$ of rank $r$ and determinant $\xi$ with a parabolic
structure over $D$. We assume the parabolic weights to be concentrated \cite{AG} (its definition is
recalled in Section \ref{sec:preliminaries}). Let $\Gamma_r$ denote the group of holomorphic
line bundles $L$ on $X$ such that $L^{\otimes r}\,=\,{\mathcal O}_X$. This group $\Gamma_r$ acts on
$P\mathcal{M}_\xi$; the action of any $L\, \in\, \Gamma_r$ sends any parabolic vector bundle $E_*$
to $E_*\otimes L$. Our aim is to determine Chen-Ruan cohomology of the corresponding orbifold.

\section{Parabolic bundles}\label{sec:preliminaries}

We recall basics on parabolic vector bundles and describe the moduli
spaces of stable parabolic bundles.

Let $X$ be a compact Riemann surface, of genus $g \,\geq\, 2$, and
$D \,=\, \{x_1,\, \dotsc,\, x_m\} \subset X$. Let $E$ be a holomorphic
vector bundle on $X$. A \emph{quasi-parabolic structure} on $E$ is a strictly decreasing flag
of linear subspaces in the fiber $E_x$ 
\begin{equation*}
 E_x \,=\, E^1_x \,\supset\, E^2_x \,\supset\, \dotsb \,\supset\, E^k_x 
\, \supset\, E^{k+1}_x\, =\, 0
\end{equation*}
over each $x \,\in\, D$. A \emph{parabolic structure} on $E$ is
a quasi-parabolic structure on $E$ together with a sequence of
real numbers $0 \,\leq\, \alpha^x_1 \,<\, \dotsb \,<\, \alpha^x_k \,<\, 1$, which
are called the \emph{weights}. We set
\begin{equation*}
	\label{eq:a2}
 m^x_j \,=\, \dim_{\mathbb{C}}{E^j_x} - \dim_{\mathbb{C}}{E^{j+1}_x}.
\end{equation*}
The integer $k$ is called the \emph{length} of the flag and the
string of integers $(m^x_1,\, \dotsc,\, m^x_k)$ is called the \emph{type}
of the flag. We say that the flag is a \emph{full flag} if $m^x_j \,=\, 1$ for
all $1 \,\leq\, j \,\leq\, k\,=\, {\rm rank}(E)$.

A \emph{parabolic vector bundle} with parabolic structure on $D$ is a
holomorphic vector bundle $E$ together with a parabolic structure on
$E$; it will be denoted by $E_*$. For a
parabolic bundle $E_*$ the \emph{parabolic degree} is defined to be
\begin{equation*}
p\deg(E_*) \,=\, \deg(E) + \sum_{x \in D} 
\sum_{j = 1}^k m^x_j \alpha^x_j\, \in\, {\mathbb R},
\end{equation*}
where $\deg(E)$ denotes the degree of $E$, and we put
\begin{equation*}
p\mu(E_*) \,=\, \frac{p\deg(E_*)}{{\rm rank}(E)},
\end{equation*}
which is called the parabolic slope of $E_*$.

A parabolic subbundle of $E_*$ is a subbundle $F$ of $E$ together with the parabolic structure induced from $E_*$.
A parabolic bundle $E_*$ is called \emph{parabolic semistable} if
for every non-zero proper parabolic subbundle $F_*$
\begin{equation}\label{eq:a5}
p\mu(F_*) \,\leq\, p\mu(E_*),
\end{equation}
and it is called \emph{parabolic stable} if $p\mu(F_*) \,<\, p\mu(E_*)$.

The moduli space $P\mathcal{M}^\alpha(r,d)$ of semistable parabolic bundles of rank $r$ and
degree $d$ and parabolic type $\alpha$, which was constructed in \cite{MS},
is a normal projective variety. The moduli space $P\mathcal{M}$ of stable parabolic bundles is an open
and smooth subset of $P\mathcal{M}^{\alpha}(r,d)$.

For a fixed rank $r$, a full flag systems of weights $\alpha \,=\, \{\alpha_1^x,\, \cdots,\, \alpha_r^x\}_{x \in D}$
is said to be concentrated if $$\alpha_r^x - \alpha_1^x \,<\, \frac{4}{mr^2}$$ for all $x \,\in\, D$.

In what follows, $r$ will denote a prime number.
Let $\xi$ be a fixed line bundle of degree $d$ such that $(d,\, r)\, =\,1$. Let
$P\mathcal{M}_\xi$ denote the moduli space of semistable parabolic
vector bundles $E_*$ over $X$ of rank $r$ with full flag concentrated weight system $\alpha$ together with an
isomorphism $\bigwedge^rE \,\cong\, \xi$. Then all $E_*\, \in\, P\mathcal{M}_\xi$ is stable, and
$P\mathcal{M}_\xi$ is a smooth complex projective variety.

We fix the parabolic weight system $\alpha$ throughout. The rank $r$ is also
fixed throughout.

Let $\Gamma$ be the subgroup of ${\rm Pic}^0(X)$ consisting of all holomorphic line bundles $L$ on $X$
such that $L^{\otimes r}\,=\, {\mathcal O}_X$. Then $\Gamma$ is isomorphic to $(\mathbb{Z}/r\mathbb{Z})^{\oplus 2g}$. Note that
for every $L \,\in\, \Gamma$, we have 
\begin{equation*}
\bigwedge\nolimits^r(E\otimes L) \,=\, (\bigwedge\nolimits^r E) \otimes L^{\otimes r}\, =\, \xi.
\end{equation*}
For each $x \,\in\, D$ we have a filtration
\begin{equation*}
(E\otimes L)_x\, =\, E^1_x \otimes L_x \,\supset\, E^2 \otimes
L_x \,\supset\, \cdots \,\supset\, E^r_x \otimes L_x
\end{equation*}
given by the parabolic structure of $E_*$ at $x$. The resulting
parabolic bundle with $E\otimes L$ as the underlying vector bundle
will be denoted by $E_* \otimes L$. It is to be noted that
$E_* \otimes L$ is the parabolic tensor product of $E_*$ with the line
bundle $L$ equipped with the trivial parabolic structure (see
\cite{Bis} for the parabolic tensor product). Let
\begin{equation}\label{eq:a22}
 \widetilde{\phi}_L \,\, \colon\,\, P\mathcal{M}_\xi \,\longrightarrow\, P\mathcal{M}_\xi\, ,\ \
E_*\, \longmapsto\, E_* \otimes L
\end{equation}
be an automorphism.
This gives an action $\widetilde{\phi}$ of the group $\Gamma$ on
$P\mathcal{M}_\xi$
\begin{equation*}
 \widetilde{\phi}(L,E_*) \,= \,\widetilde{\phi}_L(E_*).
\end{equation*}
The quotient space $$Y \,=\, P\mathcal{M}_\xi/\Gamma$$ is a
smooth orbifold. Our aim is to compute the Chen-Ruan cohomology of the
orbifold $Y$.

\section{Fixed point sets} 

We continue with the notation of the previous section. For any
$L\,\in\, \Gamma$, set
\begin{equation}\label{eq:a25}
PS(L) \,:=\, (P\mathcal{M}_\xi)^{\widetilde{\phi}_L} \,=\,
\{E_* \,\in\, P\mathcal{M}_\xi \,\,\mid\,\, \widetilde{\phi}_L(E_*) \,=\, E_*\}
\,\subset\, P\mathcal{M}_\xi.
\end{equation}
Then $PS(L)$ is a compact complex manifold, but it need not be connected.

Let $\mathcal{M}_\xi$ denote the moduli space of stable vector bundles 
$E$ on $X$ of rank $r$ with $\det(E)\,=\, \xi$. It is a smooth
projective variety. From \cite[Proposition 2.6]{AG}, we have a well defined
forgetful morphism
\begin{equation}
	\label{eq:a26}
 \gamma \,\colon\, P\mathcal{M}_\xi \,\longrightarrow\, \mathcal{M}_\xi
\end{equation}
that sends a parabolic vector bundle $E_*$ to its underlying vector
bundle $E$. This map is surjective, and the fiber of $\gamma$ over any $E$
is the product of full flags of $E_x$ with $x \,\in\, D$ \cite[Proposition 2.6]{AG}.
Further, we have an action $\phi$ of $\Gamma$ on $\mathcal{M}_\xi$ given by
\begin{equation*}
 \phi(L,V) \,=\, V \otimes L
\end{equation*}
for $L \,\in\, \Gamma$ and $V \,\in\, \mathcal{M}_\xi$. This action gives an automorphism
\begin{equation}
\label{eq:a27.5}
\phi_L \,:\, \mathcal{M}_\xi \,\longrightarrow\, \mathcal{M}_\xi\, , \ \ V \,\longmapsto\, V \otimes L.
\end{equation}
The morphism $\gamma$ in \eqref{eq:a26} is evidently $\Gamma$-equivariant.

Consider
\begin{equation*}
	\label{eq:a28}
 S(L) \,:=\, (\mathcal{M}_\xi)^{\phi_L} \,=\, \{E \,\in\, \mathcal{M}_\xi\,\,\mid\,\, \phi(L,E) \,=\, E\}
\,\subset\, \mathcal{M}_\xi,
\end{equation*}
which is a smooth compact complex manifold. As $\gamma$ is
$\Gamma$-equivariant, we have
\begin{equation*}
	\label{eq:a29}
 \gamma(PS(L)) \,\subset\, S(L).
\end{equation*}
For a description of $S(L)$ see \cite[p.~499, Lemma 2.1]{BP1}.
Let 
\begin{equation}
\label{eq:a29.1}
\Psi_L\,\, : \,\,PS(L) \,\longrightarrow\, S(L)
\end{equation}
be the restriction of $\gamma$.

Take a nontrivial line bundle $L \,\in\, \Gamma$. Fix a nonzero
holomorphic section $$s \,:\, X \,\longrightarrow\, L^{\otimes r}.$$ Define 
\begin{equation*}
Y_L \,:=\, \{z \,\in\, L\,\,\mid\,\, z^{\otimes r} \,\in\, \mathrm{Im}(s)\} \,\subset\, L.
\end{equation*}
Let $$\pi_L \,:\, Y_L \,\longrightarrow\, X$$ be the restriction of the natural projection
$L \,\longrightarrow\, X$.
Since order $r$ of $L$ is a prime number, $Y_L$ is an irreducible curve and $\pi_L$ is an unramified 
covering of degree $r$. More precisely, $Y_L$ is a $\mu_r$-bundle over $X$, where
$\mu_r \,:= \,\{ a \,\in \,{\mathbb C}\,\mid\, a^r \,=\, 1\}$.

\begin{lemma}
\label{lem:l3}
Let $L \,\in\, \Gamma$ be a nontrivial line bundle on $X$.
\begin{enumerate}
\item The map $\Psi_L$ in \eqref{eq:a29.1} is surjective.
\item The map $\Psi_L$ is an isomorphism on each connected component of $PS(L)$. The number of connected
components of $PS(L)$ is $(r!)^m$.
\end{enumerate}
\end{lemma}

\begin{proof}
Let $$\mbox{Prym}_\xi \,\subset\, {\rm Pic}^1(Y_L)$$
be the locus of all line bundles 
$\eta \longrightarrow Y_L$ such that 
$\bigwedge^r {\pi_L}_* \eta = \xi$.
The Galois group ${\rm Gal}(\pi_L) \,=\, \mathbb{Z}/ r\mathbb{Z}$ acts on 
$\mbox{Prym}_\xi$; the action of any
\begin{equation}\label{eq:a32}
\sigma\, \in\, {\rm Gal}(\pi_L)
\end{equation}
sends any $\eta\,\in\, \mbox{Prym}_\xi$ to $\sigma^* \eta$. We have 
\begin{equation}
\label{eq:a33}
S(L) \,=\, \mbox{Prym}_{\xi} / {\rm Gal}(\pi_L)
\end{equation}
(see \cite[Lemma 2.1]{BP1}).

For any $\eta \,\in\, \mbox{Prym}_{\xi}$, we have 
$$\pi_L^* \pi_{L*} \eta \,=\, \bigoplus_{\sigma \in {\rm Gal}(\pi_L)} \sigma^* \eta.$$
The fiber of ${\pi_L}_* \eta$ over $x_i \in D$ has the 
following decomposition
\begin{equation}
\label{eq:a34}
({\pi_L}_* \eta)_{x_i} \,= \,\bigoplus_{z \in \pi_L^{-1}(x_i)} \eta_z\,=\, \eta_{y_i} \oplus \eta_{\sigma(y_i)}
\oplus \eta_{\sigma^2(y_i)} \oplus \cdots \oplus \eta_{\sigma^{r-1}(y_i)},
\end{equation}
where $y_i \,\in\, \pi_L^{-1}(x_i)$ is a fixed point and $\sigma$ is a nontrivial automorphism as in \eqref{eq:a32}.
The filtration 
\begin{equation}
\label{eq:a34.1}
\eta_{y_i} \,\subset\, \eta_{y_i} \oplus \eta_{\sigma(y_i)} \,\subset\, \cdots \,\subset\,
\eta_{y_i} \oplus \eta_{\sigma(y_i)}
\oplus \eta_{\sigma^2(y_i)} \oplus \cdots \oplus \eta_{\sigma^j(y_i)} \,\subset\, \cdots \,\subset\,
({\pi_L}_* \eta)_{x_i}
\end{equation}
defines a parabolic structure on the vector bundle 
${\pi_L}_* \eta$ over $x_i$. Note that we have choices for the above filtration;
the direct summands in \eqref{eq:a34} can be permuted, so there are exactly $r!$ possible filtration.

Thus, for any $\eta\, \in\, S(L)$ making the above choice a parabolic bundle
$$E_* \,\longrightarrow\, X$$
is constructed; so $E \,=\, {\pi_L}_* \eta$.

We will show that $E_* \,\in\, PS(L)$, that is, 
${\pi_L}_* \eta$ is canonically isomorphic to 
$({\pi_L}_* \eta) \otimes L$.
The Riemann surface $Y_L$ lies in
$L \setminus \{0_X\}$, where $0_X \subset L$ is the image of the zero section of $L$. Therefore 
$\pi_L^* L$ has a canonical trivialization. Let 
$$u\,:\, Y_L \,\longrightarrow\, \pi_L^* L$$ be the tautological nonzero section giving the
trivialization of $\pi_L^* L$. Then
we have an isomorphism 
$$\eta \,\,\xrightarrow{\,\,\,\otimes u\,\,\,}\,\, \eta \otimes \pi_L^* L$$
defined by the tensor product with $u$. From the projection formula we get an isomorphism
\begin{equation}
\label{eq:a35}
\rho\,\,:\,\,{\pi_L}_*\eta\,\longrightarrow\, {\pi_L}_* (\eta\otimes\pi_L^* L)\, =\,({\pi_L}_* \eta) \otimes L.
\end{equation}
This isomorphism $\rho$ evidently preserves the decompositions of 
$ ({\pi_L}_* \eta)_{x_i}$ and $(({\pi_L}_* \eta)\otimes L)_{x_i}$ (see \eqref{eq:a34}).
This proves the first part of the lemma.

For the isomorphism $\rho$ in \eqref{eq:a35}, 
$$\rho_{x_i}(V) \,=\, V \otimes L_{x_i}$$
for some subspace $V \,\subset\, ({\pi_L}_* \eta)_{x_i}$
if and only if $V$ is the direct sum of some direct summands in \eqref{eq:a34}.
As mentioned before, for the parabolic structures on ${\pi_L}_* \eta$ over each $x_i \,\in\, D$,
we have $r!$ choices for the parabolic filtration. Note that each choice gives a copy of $S(L)$.
Therefore, $PS(L)$ is the disjoint union of copies of 
$S(L)$, and the copies are parametrized by the finite 
set
$$\prod_{i = 1}^{m} \mbox{Sym}({\pi_L}^{-1}(x_i))$$
which has cardinality $(r!)^m$.
This completes the proof.
\end{proof}

\begin{corollary}
\label{cor:1}
Let $L \,\in\, \Gamma$ be a nontrivial line bundle. Then
the cohomology algebra $H^*(PS(L),\, {\mathbb Q})$ is isomorphic to $H^*(S(L),\, {\mathbb Q})^{(r!)^m}$.
\end{corollary}

\section{Action on the tangent bundle}

The full length (same as complete) flag variety
of a finite dimensional complex vector space $W$ will be denoted by $\mathcal{F}(W)$.
The inverse image of any $E \,\in\, \mathcal{M}_{\xi}$
for the map $\gamma$ defined in \eqref{eq:a26} is 
$$\gamma^{-1}(E) \,=\, \prod_{x_i \in D} \mathcal{F}(E_{x_i}).$$
Therefore, elements of $P \mathcal{M}_{\xi}$
are of the form $(E;\, f_1, \,\cdots, \,f_m)$, where $E \,\in\, \mathcal{M}_{\xi}$ and 
$(f_1,\, \cdots,\, f_m) \,\in\, \prod_{x_i \in D} \mathcal{F}(E_{x_i})$.
Note that any $f_i \,\in\, \mathcal{F}(E_{x_i})$ corresponds to a filtration of the following type
\begin{equation}
\label{eq:a36}
f_i \,:=\, \{E_{x_i} \,=\, E^1_{x_i} \,\supset\, E^2_{x_i} \,\supset\,
\cdots \,\supset \,E^r_{x_i} \,\supset \,E^{r+1}_{x_i} \,=\, 0\}.
\end{equation}

Recall that the tangent space of the flag variety $\mathcal{F}(E_{x_i})$ is $\text{End}_{\mathbb{C}}{(E_{x_i})}/ 
\text{End}^0_{\mathbb{C}}{(E_{x_i})}$, where $\text{End}_{\mathbb{C}}^0(E_{x_i})$ is the space of flag preserving
$\mathbb{C}$-linear endomorphisms of $E_{x_i}$.

If $E_* \, \in \, PS(L)$, then using
\eqref{eq:a34}, the tangent space at any point of the flag variety $\mathcal{F}(E_{x_i})$ canonically
decomposes as 
$$
\text{End}_{\mathbb{C}}{(E_{x_i})}/
\text{End}^0_{\mathbb{C}}{(E_{x_i})}\,=\, \bigoplus_{0 \leq j < k \leq r-1} \,
{\rm Hom}(\eta_{\sigma^j (y_i)}, \, \eta_{\sigma^k(y_i)}),$$
where $E_* \, = \,{\pi_{L}}_*(\eta)$, and $\eta \, \in\, \mbox{Prym}_\xi$ (see the first part of the Lemma \ref{lem:l3}).

For $i \,\in\, \{1,\,2,\, \cdots,\, m\}$, let
$$V_i \,\longrightarrow \,P\mathcal{M}_{\xi}$$
be the rank $\frac{1}{2} r(r-1)$ vector bundle whose fiber over 
any $(E;\, f_1, \cdots, f_m) \,\in\, P \mathcal{M}_{\xi}$ is the vector space 
\begin{equation}
\label{eq:a36.1}
V_i(E;\, f_1,\, \cdots,\, f_m) \,=\, \text{End}_{\mathbb{C}}{(E_{x_i})}/ \text{End}_{\mathbb{C}}^0{(E_{x_i})}.
\end{equation}

Let $TP\mathcal{M}_\xi$ (respectively, $T\mathcal{M}_\xi$) denote the holomorphic tangent bundle of
$P\mathcal{M}_\xi$ (respectively, $\mathcal{M}_\xi$). They fit in the following short exact sequence of vector bundles 
\begin{equation}
\label{eq:a37}
0 \,\longrightarrow \, \bigoplus_{i =1}^m V_i\,\longrightarrow\, TP\mathcal{M}_\xi \,\xrightarrow{\,d\gamma\,}\,
\gamma^* T\mathcal{M}_{\xi} \,\longrightarrow \, 0,
\end{equation}
where $d\gamma$ is the differential of the map $\gamma$.

Take any $L \,\in\, \Gamma \setminus \{\mathcal{O}_X\}$. The automorphism
$\widetilde{\phi}_L$ in \eqref{eq:a22} induces an automorphism 
$$d\widetilde{\phi}_L \,:\, TP\mathcal{M}_\xi \,\longrightarrow\, TP\mathcal{M}_\xi$$ of tangent bundles
over the map $\widetilde{\phi}_L$. Next, $d\widetilde{\phi}_L$ induces an automorphism over $\widetilde{\phi}_L$
\begin{equation}\label{eq:a37.1}
A^i\,\, :\,\, V_i \,\longrightarrow\,\, V_i,
\end{equation}
for every $i\, =\, 1,\, \cdots,\, m$, such that the following diagram of
homomorphisms
\begin{equation}
\label{eq:cd1}
\xymatrix@C=4em{
0 \ar[r] & \bigoplus_{i =1}^m V_i \ar[d]^{\oplus A^i} \ar[r] & TP\mathcal{M}_\xi
\ar[d]^{d\widetilde{\phi}_L} \ar[r]^{d\gamma} &\gamma^* T\mathcal{M}_{\xi} \ar[d]^{d\phi_L} \ar[r] & 0 \\
0 \ar[r] &\bigoplus_{i =1}^m V_i \ar[r] & 
TP\mathcal{M}_\xi \ar[r]^{d\gamma} &\gamma^* T\mathcal{M}_{\xi} \ar[r] & 0 }
\end{equation}
commutes.

Take any parabolic vector bundle $E_* \,\in\, PS(L)$. Let
$T_{E_*} (P\mathcal{M}_\xi)$ denote the tangent space at $E_*$, and let
$$d\widetilde{\phi}_L(E_*) \,\,: \,\, T_{E_*}(P\mathcal{M}_\xi)
\,\longrightarrow\, T_{E_*}(P\mathcal{M}_\xi)$$
be the differential of the map $\widetilde{\phi}_L$ at $E_*$. The diagram in \eqref{eq:cd1} gives the
commutative diagram
\begin{equation}
\label{eq:cd2}
\xymatrix@C=3em{
0 \ar[r] & \bigoplus_{i =1}^m V_i(E_*) \ar[d]^{\oplus A^i_{E_*}} \ar[r] & T_{E_*}(P\mathcal{M}_\xi)
\ar[d]^{d\widetilde{\phi}_L(E_*)} \ar[r]^{d\gamma(E_*)} & T_E(\mathcal{M}_{\xi}) \ar[d]^{d\phi_L(E)} \ar[r] & 0 \\
0 \ar[r] &\bigoplus_{i =1}^m V_i(E_*) \ar[r] & 
T_{E_*}(P\mathcal{M}_\xi) \ar[r]^{d\gamma(E_*)} & T_E(\mathcal{M}_{\xi}) \ar[r] & 0 \,\,.}
\end{equation}
Since $r$-fold composition of $\widetilde{\phi}_L$ yields
$$\widetilde{\phi}_L \circ \widetilde{\phi}_L\circ\cdots 
\circ \widetilde{\phi}_L\, =\, \widetilde{\phi}_{L^r} \,=\, {\rm Id}_{P\mathcal{M}_\xi},$$
$d\widetilde{\phi}_L(E_*)$ is a nontrivial automorphism of order $r$. Therefore, the set of eigenvalues of
$d\widetilde{\phi}_L(E_*)$ is $$\mu_r \,\,:=\,\, \{1,\, t,\, t^2,\,\cdots,\, t^{(r-1)}\},$$
the group of $r$-th roots of unity. 
Our aim is to compute the multiplicity of each eigenvalue in $\mu_r$ of the linear operator $d\widetilde{\phi}_L(E_*)$. 

Note that $d\phi_L(E)$ and $A^i_{E_*}$, $i\,=\,1, \,\cdots,\, m$ (see \eqref{eq:cd2}), are 
also nontrivial automorphisms of order $r$. For any eigenvalue $\nu \,\in \,\mu_r$, the 
multiplicity of $\nu$ for the operator $d\widetilde{\phi}_L(E_*)$ is evidently the sum of 
the multiplicities of $\nu$ for the operators $\bigoplus_{i = 1}^m A^i_{E_*}$ and 
$d\phi_L(E)$.
 
The multiplicity of every $\nu \,\in\, \mu_r$ for the operator $d\phi_L(E)$ has been 
computed in \cite[Proposition 3.1]{BP1}. It is $r(g-1)$ if $\nu \,\in \,\mu_r \setminus 
\{1\}$, and the multiplicity of $1\,\in \,\mu_r$ is $(r-1)(g-1)$. So we need to determine 
the multiplicities of the eigenvalues for the operator $A^i_{E_*}$ for every 
$i\,=\,1,\,\cdots,\, m$.

Recall from Lemma \ref{lem:l3} that $PS(L)$ is $(r!)^m$ copies of $S(L)$, and they arise from the choices of 
filtration in \eqref{eq:a34.1}, given the decomposition in \eqref{eq:a34}; for each point $x_i\in\, D$, there
are $r!$ possible filtrations and there are $m$ points in $D$.
Fix $y_i\, \in\, \pi_L^{-1}(x_i)$; trivialize the fiber $L_{x_i}$ using $y_i\,\in\, L_{x_i}$ (recall that
$Y_L\, \subset\, L$). Let
\begin{equation}\label{et}
\lambda\, \in\, \mu_r\setminus \{1\}
\end{equation}
be such that $\sigma(y_i)\,=\, \lambda\cdot y_i\, \in\, L_{x_i}$. Note that $\lambda$ depends on $\sigma\,\in\, {\rm Gal}(\pi_L)$,
but it does not depend on $y_i$. Then for any $E\, \in\, S(L)$, the isomorphism
$$
E\, \longrightarrow\, E\otimes L
$$
acts on the direct summand $\eta_{\sigma^j(y_i)}$ in \eqref{eq:a34} as multiplication by $\lambda^j$,
where $\lambda$ is the element in \eqref{et}. Using this,
it is straightforward to deduce the following:

\begin{lemma}\label{lemem}
Let ${\mathcal C}\, \subset\, PS(L)$ be one of the 
$(r!)^m$ components of $PS(L)$ (recall that each component is a copy of $S(L)$).
The set of eigenvalues of $ \bigoplus_{i =1}^m A^i_{E_*}$ (see \eqref{eq:cd2}), for
any $E_*\, \in\, {\mathcal C}$, is
$\mu_r\setminus \{1\}$. For all $1\, \leq\, i\, \leq\, m$, there is an element 
$$
s_i\, \in\, \,\mu_r\setminus \{1\}
$$
(that depends on ${\mathcal C}$) such that for every $1\, \, \leq\, c\, \leq\, r-1$, the
multiplicity of the eigenvalue $s^c_i$ of $A^i_{E_*}$ is $r-c$.
\end{lemma}

\begin{proof}
It is a straightforward computation; the details are omitted.
\end{proof}

\begin{proposition}\label{prop:0}
Let ${\mathcal C}\, \subset\, PS(L)$ be one of the 
$(r!)^m$ components of $PS(L)$.
The set of eigenvalues of $d\widetilde{\phi}_L(E_*)$, for any $E_* \,\in \,\mathcal{C}$, is $\mu_r$.
For all $1\, \leq\, i\, \leq\, m$, there is an element 
$$
s_i \,\in\, \mu_{r}\setminus \{1\}
$$
such that the multiplicity of the eigenvalue $s\,\in\, \mu_r\setminus \{1\}$ of $d\widetilde{\phi}_L(E_*)$ is
$r(g-1)+ \sum_{i=1}^m(r-c_i)$, where $s\,=\, (s_i)^{c_i}$; the multiplicity of the eigenvalue $1$ of
$d\widetilde{\phi}_L(E_*)$ is $(r-1)(g-1)$.
\end{proposition}

\begin{proof}
Considering the diagram in \eqref{eq:cd2} it follows immediately that the collection of
eigenvalues of $d\widetilde{\phi}_L(E_*)$ (with multiplicities) is the union of the eigenvalues
of $\bigoplus_{i=1}^m A^i_{E_*}$ and $d\phi_L(E)$. It was noted above that the 
multiplicity of the eigenvalue $\nu \,\in\, \mu_r \setminus \{1\}$ for $d\phi_L(E)$
is $r(g-1)$, and the multiplicity of the eigenvalue $1$ for $d\phi_L(E)$ is $(r-1)(g-1)$. 
Therefore, the proposition follows from Lemma \ref{lemem}
\end{proof}

\begin{remark}
Fix any $i$ with $1\, \leq\, i\, \leq\, m$. For each component ${\mathcal C}\, \subset\, PS(L)$, consider
$s_i\,\in\, \mu_{r}\setminus \{1\}$ in Proposition \ref{prop:0}. We note that every element
$s \, \in\, \mu_{r}\setminus \{1\}$ repeats $\frac{(r!)^m}{r-1}$ in this collection.
\end{remark}

For any $L\, \in\, \Gamma \setminus \{\mathcal{O}_X\}$ and $E_* \,\in\, PS(L)$, the degree shift at $E_*$ for $L$ 
is defined by
\begin{equation*}
	\label{eq:a38}
\pi(L,E_*) \,:=\, \sum_j m_j b_j,
\end{equation*}
where $\exp(2\pi \sqrt{-1}b_j)$, $0 \,\leq\, b_j \,<\,1$, are the eigenvalues of $d\widetilde{\phi}_L(E_*)$,
and $m_j$ is the multiplicity of $\exp(2\pi \sqrt{-1}b_j)$. 

For an integer $b$, let $0\, \leq\, [b]_r\, \leq\, r-1$ be such that
$b\,=\, kr +[b]_r$ with $k\, \in\, \mathbb Z$.

As a corollary of Proposition \ref{prop:0}, we get the following.

\begin{corollary}\label{cor:c1}
With the notations used in Proposition \ref{prop:0},	
assume $s_i = t^{l_i}$, $1 \leq l_i\leq r-1$ and $ 1 \leq i \leq m$, where $t\, =\, \exp(\frac{2\pi \sqrt{-1}}{r})$. Then for any $L \,\in\, \Gamma \setminus \{\mathcal{O}_X\}$, the degree shift at $E_*$ for $L$ is $$\pi(L) \, := \,
\pi(L, E_*)\,=\, \frac{1}{r}\sum_{i=1}^m
\left[\sum_{k=1}^{r-1} (r-[k c_i]_r)[kl_ic_i]_r \right] + \frac{r(r-1)(g-1)}{2}.$$
\end{corollary}

\begin{proof}
As in Proposition \ref{prop:0}, $s\,=\, (s_i)^{c_i} = t^{l_i c_i}$
has multiplicity $\sum_{i=1}^m(r-c_i) + r(g-1)$ for 
$d\widetilde{\phi}_L(E_*)$ and the corresponding $b_j$
will be $\frac{ [l_i c_i]_r}{r}$.

Similarly, $s^k \,=\, (s_i)^{k c_i}\, =\, t^{k l_i c_i}$
has multiplicity $\sum_{i=1}^m(r- [kc_i]_r) + r(g-1)$ for 
$d\widetilde{\phi}_L(E_*)$ and the corresponding $b_j$ will be $\frac{[k l_i c_i]_r}{r}$.
\end{proof}

\begin{remark}
For $r \,=\, 3$ and $m\,=\, 1$, the degree shift at $E_*$ for $L$ is either $\frac{4}{ 3} + 3(g-1)$ or 
$\frac{5}{3} + 3(g-1)$ depending on the component in which $E_*$ lies.
\end{remark}

\section{Chen-Ruan cohomology of the moduli space}

Let $S$ and $T$ be two topological spaces, and $\delta \,\in\, H^b(S \times T,\, {\mathbb Q})$. Then
the induced linear map
\begin{equation*}
\sigma(\delta) \,:\, H_c(S ,\, {\mathbb Q}) \,\longrightarrow\, H^{b-c}( T,\, {\mathbb Q})
\end{equation*}
is known as the \emph{slant product}.

The moduli space $P\mathcal{M}_\xi$ is a fine moduli space, and there exists a universal parabolic bundle 
$\mathcal{U}\,\longrightarrow\, X\times P\mathcal{M}_\xi$ \cite[Proposition 3.2]{BY}. Any two such universal bundles
differ by tensoring with a line bundle pulled back from $P\mathcal{M}_\xi$. For $2 \,\leq\, k \,\leq\, r$, let
$a_k({\mathbb P}(U))\, \in\, H^{2k}(\mathcal{U},\, {\mathbb Q})$ be the 
characteristic classes of the projective bundle ${\mathbb P}(U)$, where $U$ is the vector bundle underlying the 
parabolic bundle $\mathcal{U}$. Since any two universal parabolic bundles
differ by tensoring with a line bundle, it follows 
that $a_k({\mathbb P}(U))$ is independent of the choice of the universal bundle (see \cite[Remark 2.1]{BR}). So $a_k({\mathbb P}(U))$ induces linear maps
\begin{equation*}
\sigma_j(a_k({\mathbb P}(U))) \,:\, H_j(X,\,{\mathbb Q})\,\longrightarrow\, H^{2k-j}(P \mathcal{M}_\xi ,\,{\mathbb Q})
\end{equation*}
for $j\, =\, 0,\,1,\,2$. Then by \cite[Theorem 1.5]{BR}, the cohomology algebra $H^*(P\mathcal{M}_\xi,\,{\mathbb Q})$
is generated by the Chern classes $c_j(\mathcal{H}om(U^l_x,\,U^{l-1}_x))$ ($x\,\in\, D$) and the images of
$\sigma(c_1(U))$ and $\sigma_i(a_k({\mathbb P}(U)))$ ($2\,\leq\, k \,\leq\, r$,\, $0\,\leq\, i\,\leq\, 2$).
Here $U_x^l$'s are obtained from the parabolic structure of $U$ at the point $x\,\in\, D$.
The above generators are independent of the universal parabolic bundle $\mathcal{U}$ chosen. 

Consider the action of $\Gamma$ on $P\mathcal{M}_\xi$ induced by $\widetilde{\phi}$, and let
\begin{equation*}
\chi\,\,:\,\, P\mathcal{M}_\xi\,\longrightarrow\, P\mathcal{M}_\xi/\Gamma
\end{equation*}
be the quotient map. The pullback map on the cohomologies for $\chi$ will be denoted by
$\chi^*$. Using the same technique as in \cite[Proposition 4.1]{BD}, we get the following:

\begin{proposition}\label{prop:1}
The homomorphism $$\chi^* \,:\, H^*(P \mathcal{M}_\xi / \Gamma,\,{\mathbb Q})\,\longrightarrow\, H^*(P\mathcal{M}_\xi,\,
{\mathbb Q})$$
is an isomorphism.
\end{proposition}

We now describe the Chen-Ruan cohomology algebra of $P\mathcal{M}_\xi/\Gamma$.

The Chen-Ruan cohomology
group of $P\mathcal{M}_\xi/\Gamma$, by definition, is
\begin{equation*}
H^j_{CR}(P\mathcal{M}_\xi/\Gamma,\,{\mathbb Q})\,:=\, \bigoplus_{L \in \Gamma} H^{j-2\pi(L)}(PS(L)/\Gamma,\,
{\mathbb Q})\, ,\ \ j \,\geq\, 0,
\end{equation*}
where the degree shift $\pi(L)$ is given by Corollary
\ref{cor:c1}. The degree shift for the trivial line bundle
$\mathcal{O}_X$ is zero. Using Corollary \ref{cor:1} and Proposition \ref{prop:1}, we get that
\begin{equation}\label{eq:a43}
H^*_{CR}(P\mathcal{M}_\xi/\Gamma,\,{\mathbb Q})\, =\,
H^{*}(P\mathcal{M}_\xi,\, {\mathbb Q}) \oplus
\left(\bigoplus_{L \in \Gamma \setminus \{\mathcal{O}_X\}} H^{*-2\pi(L)}(S(L)/\Gamma,\,{\mathbb Q})^{\oplus^{(r!)^m}}\right).
\end{equation}
The additive structure on
$H^*_{CR}(P\mathcal{M}_\xi/\Gamma,\,{\mathbb Q})$ is the unique operation on it which gives the isomorphism of the
groups in \eqref{eq:a43}. We will now give the product structure `$\cup$' on it.

Let $k \,=\, (r!)^m$. From Section \ref{sec:preliminaries} we have that
the fixed point locus $PS(L)$ is a smooth compact submanifold of
$P\mathcal{M}_\xi$ of real dimension $2(r-1)(g-1)$
(see \cite[p. 519]{BP}) having $k$ connected components which are
copies of $S(L)\,=\,\mbox{Prym}_{\xi} / {\rm Gal}(\pi_L)$. 

Let $\widetilde{\omega}$ denote the $\Gamma$-invariant differential form on
$PS(L)$ which is the pullback of the differential form $\omega$ on
$PS(L)/\Gamma$. Then
\begin{equation*}
\widetilde{\omega} \,=\, (\widetilde{\omega_1},\, \cdots,\,\widetilde{\omega_k}),
\end{equation*}
where $\widetilde{\omega_j}$ is a differential form on the $j$-th copy of
$S(L)$ in $PS(L)$.

Recall that the orbifold integration of a
$2(r-1)(g-1)$-form $\omega$ on $PS(L)/\Gamma$ is defined as 
\begin{equation*}
\int_{PS(L)/\Gamma}^{\mbox{orb}} \omega \,:=\,
\frac{1}{|\Gamma |} \int_{PS(L)} \widetilde{\omega} \,=\,
\frac{1}{|\Gamma |} \sum_{j=1}^k \int_{S(L)} \widetilde{\omega}_j,
\end{equation*}
where $|\Gamma |\, =\, r^{2g}$ is the order of the group $\Gamma$.

The real dimension of $P\mathcal{M}_\xi/\Gamma$ is
$2d \,= \,2(r^2-1)(g-1)+ m(r^2-r)$.

Let $L\,\in\, \Gamma$. For any
$\delta_L \,=\, (\delta_1,\, \cdots, \,\delta_k)\,\in\, H^{n-2\pi(L)}(S(L)/\Gamma,\, {\mathbb Q})^{\oplus k}$ and
$\beta_L\,=\, (\beta_1,\, \cdots,\, \beta_k)\,\in\, H^{2d-n-2\pi(L)}(S(L)/\Gamma,\,{\mathbb Q})^{\oplus k}$, we define
\begin{equation*}
\langle{\delta_L},\, {\beta_L}\rangle^L \,=\, \sum_{i=1}^k \int_{S(L)/\Gamma}^{\mbox{orb}}
\delta_i \wedge \beta_i.
\end{equation*}

\begin{definition}
	\label{def:crpair}
For any integer $0 \,\leq\, j \,\leq\, 2d$, the Chen-Ruan Poincar\'e pairing
\begin{equation*}
\langle -,\,-\rangle_{CR} \,\colon\, H^j_{CR}(P\mathcal{M}_\xi/\Gamma,\,{\mathbb Q}) \times
H^{2d-j}_{CR}(P\mathcal{M}_\xi/\Gamma,\,{\mathbb Q}) \,\longrightarrow\, \mathbb{Q}
\end{equation*}
is defined by
\begin{equation*}
	\label{eq:a48}
\langle{\delta},\, {\beta}\rangle_{CR} \,=\, \sum_{L \in \Gamma} \langle{\delta_L},\, {\beta_L}\rangle^L
\end{equation*}
for all $\delta \,=\, (\delta_L)_{L \in \Gamma} \,\in\, H^j_{CR}(P\mathcal{M}_\xi/\Gamma,\,{\mathbb Q})
\,=\, \bigoplus_{L \in \Gamma} H^{j-2\pi(L)}(PS(L)/\Gamma,\,{\mathbb Q})$ and $\beta \,=\, (\beta_L)_{L \in \Gamma}\,\in
\,H^{2d-j}_{CR}(P\mathcal{M}_\xi/\Gamma,\,{\mathbb Q})\,=\, \bigoplus_{L \in \Gamma} H^{2d-j-2\pi(L)}(PS(L)/\Gamma,\,
{\mathbb Q})$.
\end{definition}

Let $L_1,\, L_2 \,\in\, \Gamma$, and define $L_3 \,= \,(L_1 \otimes L_2)^\vee$. Also let $\widetilde{T}
\,=\, \cap_{i=1}^3 PS(L_i)$ and
$$\widetilde{e_i} \,=\, \widetilde{T}/\Gamma \,\longrightarrow\, PS(L_i)/\Gamma\, , \ \ i\,=\,1,\,2,\,3$$
be the canonical injections.

Let $\widetilde{\mathcal{F}}_{L_1,L_2}$ be the orbifold obstruction bundle on $\widetilde{T}/\Gamma$
(see \cite[Section 4]{BP}) and $c_{top}(\widetilde{\mathcal{F}}_{L_1,L_2})$ its top Chern class. Note that
\begin{equation*}
{\rm rank}(\widetilde{\mathcal{F}}_{L_1,L_2}) \,=\,\dim_{\mathbb R} \widetilde{T} - \dim_{\mathbb{R}} P\mathcal{M}_\xi +
\sum_{j=1}^3 \pi(L_j).
\end{equation*}
Then, for any $\delta \,\in\, H^{p}_{CR}(PS(L_1)/\Gamma,\, {\mathbb Q})$
and $\beta\,\in\, H^{q}_{CR}(PS(L_2)/\Gamma,\, {\mathbb Q})$, we define
$$\delta \cup \beta \,\in\, H^{p+q}_{CR}(PS(L_3)/\Gamma,\, {\mathbb Q})$$
by the relation
\begin{equation*}
\langle{\delta \cup \beta},\, {\psi}\rangle_{CR} \,=\,
\int_{\widetilde{T}/\Gamma}^{orb} \widetilde{e}_1^*(\delta) \wedge \widetilde{e}_2^*(\beta) \wedge
\widetilde{e}_3^*(\psi) \wedge c_{top}(\widetilde{\mathcal{F}}_{L_1,L_2})
\end{equation*}
for all $\psi \,=\, (\psi_1,\, \cdots,\, \psi_k) \,\in\, H^{2d-p-q}_{CR}(PS(L_3)/\Gamma,\,{\mathbb Q})$.
Extending the product `$\cup$' by $\mathbb{Q}$-linearity we get the product structure on
$H^{*}_{CR}(P\mathcal{M}_\xi, \,{\mathbb Q})$ turning it into a ring.

Let $T \,=\, \bigcap_{i=1}^3 S(L_i)$. This space is described in \cite{BP}. The space 
$\widetilde{T}$ is $k$ copies of $T$. The $\Gamma$-equivariant morphism $\gamma\,:\, 
P\mathcal{M}_\xi \,\longrightarrow\, \mathcal{M}_\xi$ maps $\widetilde{T}$ to $T$. Let 
$\mathcal{F}_{L_1,L_2}$ be the obstruction bundle on $T/\Gamma$. Thus we have
$$
\widetilde{\mathcal{F}}_{L_1,L_2}\big\vert_{T/\Gamma}
\,\cong\, \mathcal{F}_{L_1,L_2},
$$
where we have identified any connected component of $\widetilde{T}/\Gamma$ with $T/\Gamma$
using the isomorphism between them given by $\gamma$. It follows that 
\begin{equation*}
\langle{\delta \cup \beta},\, {\psi}\rangle_{CR} \,=\, \sum_{i=1}^k \langle{\delta_i \cup \beta_i},\,
{\psi_i}\rangle_{CR},
\end{equation*}
where the pairing on the right hand side of the equation is the non-degenerate bilinear Poincar\'e pairing for the Chen-Ruan cohomology on $\mathcal{M}_\xi/\Gamma$ (see \cite[6.20]{BP1}). The Chen-Ruan product $\delta_i \bigcup \beta_i$ for the orbifold bundle $\mathcal{M}_\xi/\Gamma$ is computed in
\cite{BP1}. We also have
\begin{equation*}
\delta\cup \beta \,=\, (\delta_1 \cup \beta_1,\, \cdots,\, \delta_k \cup \beta_k).
\end{equation*}
Moreover, if $L_1 \,=\, L_2 \,=\, \mathcal{O}_X$, then the Chen-Ruan product $\cup$ is the ordinary cup product on 
$H^{*}(P\mathcal{M}_\xi/\Gamma,\, {\mathbb Q})$.

\section*{Acknowledgement}

We thank the referee for helpful comments. The first-named author is partially supported 
by a J. C. Bose Fellowship. The School of mathematics of TIFR is supported by 
12-R$\&$D-TFR-5.01-0500.

\end{document}